\documentclass[review]{elsarticle}
\usepackage{lipsum}
\makeatletter
\def\ps@pprintTitle{%
 \let\@oddhead\@empty
 \let\@evenhead\@empty
 \def\@oddfoot{}%
 \let\@evenfoot\@oddfoot}
\makeatother
\usepackage{lineno,hyperref}
\usepackage{amsmath,amssymb,latexsym, amscd}
\usepackage{exscale, eps fig, graphics}
\usepackage{array}
\usepackage{float}
\restylefloat{table}
\usepackage{placeins}
\usepackage{longtable}
\usepackage{makecell}
\usepackage[utf8]{inputenc}
\usepackage[justification=centering]{caption}
\usepackage{graphicx}
\usepackage{caption}
\usepackage{subcaption}
\usepackage{mathtools}
\hypersetup{pdfauthor={Name}}

\DeclarePairedDelimiter\floor{\lfloor}{\rfloor}
\newtheorem{theorem}{Theorem}[section]
\newtheorem{lemma}[theorem]{Lemma}
\newtheorem{proposition}[theorem]{Proposition}

\newtheorem{conjecture}[theorem]{Conjecture}
\newtheorem{definition}[theorem]{Definition}
\newtheorem{corollary}[theorem]{Corollary}

\newenvironment{proof}{\paragraph{Proof}}{\hfill$\square$}
\newcommand{\BigO}[1]{\ensuremath{\operatorname{O}\left(#1\right)}}

\renewcommand{\geq}{\geqslant}
\renewcommand{\leq}{\leqslant}

\newcommand{\N}{\mathbb{N}}
\newcommand{\Z}{\mathbb{Z}}
\newcommand{\Q}{\mathbb{Q}}
\newcommand{\R}{\mathbb{R}}
\newcommand{\C}{\mathbb{C}}
\DeclareMathOperator{\lcm}{lcm}
\newcommand{\Mod}[1]{\ (\mathrm{mod}\ #1)}

\begin{document}

\begin{frontmatter}

\title{On the density of visible lattice points along polynomials}

\author[mymainaddress]{Sneha Chaubey}
\ead{sneha@iiitd.ac.in}

\author[mymainaddress]{Ashish Kumar Pandey}
\ead{ashish.pandey@iiitd.ac.in}

\address[mymainaddress]{Department of Mathematics, IIIT Delhi, New Delhi 110020}

\begin{abstract}
Recently, the notion of visibility from the origin has been generalized by viewing lattice points through curved lines of sights, where the family of curves considered are $y=mx^k$, $k\in\mathbb{N}$. In this note, we generalize the notion of visible lattice points for a given polynomial family of curves passing through the origin, and study the density of visible lattice points for this family. The density of visible lattice points for family of curves $y=mx^k$, $k\in\mathbb{N}$ is well understood as one has nice arithmetic interpretations in terms of a generalized gcd function, which seems to be absent for general polynomial families. 
 We pose ``Visibility density conjecture" regarding the density of visible lattice points for polynomial families passing through the origin, and show some numerical results supporting the conjecture. We obtain a lower bound on the density for a class of quadratic families. In addition, we discuss some ideas of the proof of the conjecture for the simplest quadratic family.
\end{abstract}

\begin{keyword}
visible lattice points, polynomial congruences, density
\MSC[2020] 11P21\sep 11M99 \sep 11H06
\end{keyword}

\end{frontmatter}


\section{Introduction and Main Results}\label{sec:intro}
A point $(a,b)$ in two-dimensional integer lattice, $\Z^2$, is said to be {\em visible} from the origin if there exists no other point of $\Z^2$ on the line segment joining the origin and the point $(a,b)$. The set of visible lattice points are in one-to-one correspondence with lines passing through the origin of rational slope (including infinite slope). If we restrict our attention only to points in $\N^2$, then visible lattice points in $\N^2$ are in one to one correspondence with family of lines 
\begin{equation}\label{fam:line}
\{y=q x \ | \ q\in \Q^+\}.
                        \end{equation}
We replace lines in \eqref{fam:line} by polynomials passing through the origin and define a new family as follows.

\begin{definition}\label{def:polyfam}
For a fixed vector $(a_n, a_{n-1},\dots,a_1)\in \Z^{n}$ with $a_n\neq 0$, $a_i\ge 0$ for all $1\le i\le n$, and $\gcd(a_n,a_{n-1},\dots,a_1)=1$, let
\[
\mathcal{F}(a_n, a_{n-1},\dots,a_1) :=  \{y=q(a_n x^n+a_{n-1} x^{n-1}+\cdots +a_1 x) \ | \ q\in \Q^+\}.
\]
\end{definition}

We generalize the notion of visibility for points in $\N^2$ to the family $\mathcal{F}$ in Definition~\ref{def:polyfam} as follows.

\begin{definition}\label{def:vis}
Consider the family $\mathcal{F}(a_n, a_{n-1},\dots,a_1)$ as in  Definition~\ref{def:polyfam}. A point $(a,b)$ in $\N^2$ is said to be $\mathcal{F}(a_n, a_{n-1},\dots,a_1)$-visible if there exists a $q\in \Q^+$ such that
$
b=q(a_n a^n+a_{n-1} a^{n-1}+\cdots +a_1 a)
$
and there is no other point of $\N^2$ on the curve
\[
y=q(a_n x^n+a_{n-1} x^{n-1}+\cdots +a_1 x)
\]
between the origin and $(a,b)$. We denote the set of all $\mathcal{F}(a_n, a_{n-1},\dots,a_1)$-visible points in $\N^2$ by $V(a_n, a_{n-1},\dots,a_1)$.
\end{definition}

We work with the following definition of density of a subset $S$ of $\N^2$.

\begin{definition}\label{def:density}
The density of a set $S\subseteq\N^2$ is
\[
\operatorname{dens(S)}=\lim_{N\to \infty} \frac{|S\cap R_N|}{N^2}
\]
provided the limit exists, where $N\in \N$, $|\cdot|$ denotes the cardinality of a set, and 
\begin{equation}\label{eq:RN}
R_N=\{(a,b)\in \Z^2 \ | \ 1\le a,b\le N\}. 
\end{equation}  
\end{definition} 

The set $V(1)$ of $\mathcal{F}(1)$-visible points is the usual set of visible lattice points in $\N^2$ and has been studied extensively from various perspectives. The density of $V(1)$ is $1/\zeta(2)=6/\pi^2$, where $\zeta(s)$ is the classical Riemann-zeta function \cite{Apostol1976}. The problems concerning counting and distribution of visible lattice points in suitable planar and smooth convex domains have been investigated by various authors (see \cite{Baker2010PrimitiveDomains,Boca2000DistributionOrigin,Huxley1996PrimitiveBy} and references therein). The topological and spectral properties of visible lattice points have been studied in \cite{Baake2015, Baake2000}. The set $V(1)$ is an example of a  weak model set, which has positive topological entropy but pure point dynamical and diffraction spectrum \cite{Baake2015, Baake2017OnDensity}. It has also been explored from a probabilistic point of view where the authors compute the asymptotic distribution of visible lattice points visited by a certain random walker (see \cite{Cilleruelo2019} and references therein). 

The density of $V(1,0,\dots,0)$ is known to be $1/\zeta(n+1)$ \cite{Goins2018}. The proof uses the following characterization of $V(1,0,\dots,0)$
\[
V(1,0,\dots,0)=\{(a,b)\in \N^2 \;|\; {\gcd}_n(a,b)=1 \}
\]
where ${\gcd}_n$ is a generalized $\gcd$ function defined by
\[
{\gcd}_n(a,b)=\max\{\ell\in \N \;|\;\ell\mid a \text{ and } \ell^n\mid b\}.
\]
In particular, for $n=1$, $V(1)$ is characterized by the usual $\gcd$ function. This arithmetic characterization which enables one to answer various distribution questions for sets $V(1, 0, \dots, 0)$ is unavailable for sets $V(a_n,a_{n-1},\dots,a_1)$. This makes the problem of determining the density non-trivial in this case. Nevertheless, this motivates us to look at the set
\[
S(a_n,a_{n-1},\dots,a_1)=\{(a,b)\in \N^2 \;|\; \gcd(a_n a^n+a_{n-1} a^{n-1}+\cdots +a_1 a, b)=1\},
\]
and its association with $V(a_n,a_{n-1},\dots,a_1)$. To this effect, we have the following result.

\begin{theorem}\label{thm:sub}
For a fixed vector $(a_n, a_{n-1},\dots,a_1)\in \Z^{n}$ with $a_n\neq 0$, $a_i\ge 0$ for all $1\le i\le n$, and $\gcd(a_n,a_{n-1},\dots,a_1)=1,$
\[
S(a_n,a_{n-1},\dots,a_1)\subseteq V(a_n,a_{n-1},\dots,a_1)\cap S(1).
\]
Moreover,
\[
\operatorname{dens}(S(a_n,a_{n-1},\dots,a_1)=\prod_{p \text{ prime}}\Big(1-\frac{f_{a_n,a_{n-1},\dots,a_1}(p)}{p^2}\Big),
\]
where
\begin{align}
    &f_{a_n,a_{n-1},\dots,a_1}(\ell):=|\{1\le d\le \ell \ | \ a_n d^n+a_{n-1}d^{n-1}+\cdots+a_1d\equiv 0 \Mod{\ell}\}|. \label{fd}
\end{align} 
\end{theorem}

Theorem~\ref{thm:sub} provides a lower bound on $\operatorname{dens}(V(a_n,a_{n-1},\dots,a_1))$. The function in \eqref{fd} can be explicitly calculated for the special case $(a_n,a_{n-1},\dots,a_1)=(1,k)$, $k\ge 2$ and we obtain the following corollary.

\begin{corollary}\label{cor:densS1k}
\[
\operatorname{dens}(V(1,k))\ge\prod\limits_{p \text{ prime}} \Big(1-\frac{2}{p^2}\Big)\prod_{\substack{p|k\\ p \text{ prime}}}\Big(1+\frac{1}{p^2-2}\Big), \quad k\ge 2.
\]
\end{corollary}

Additionally, we run some numerical experiments to estimate\\ $\operatorname{dens}(V(a_n, a_{n-1},\dots,a_1))$ which motivates us to pose the following conjecture.

\begin{conjecture}[Visibility density conjecture]\label{conj}
For $(a_n, a_{n-1},\dots,a_1)\neq (1, 0,\dots,0)$,
\[
\operatorname{dens}(V(a_n, a_{n-1},\dots,a_1))=1.
\]
\end{conjecture}

The proofs of Theorem~\ref{thm:sub} and Corollary~\ref{cor:densS1k} are presented in Section~\ref{sec:lower}. In Section~\ref{sec:numerical}, some numerical results are presented in support of Conjecture~\ref{conj} and we discuss the Conjecture~\ref{conj} for the special case $V(1,1)$ in detail in Section~\ref{sec:f11}. 

\section*{Notation}
We employ the Landau-Bachmann ``Big Oh" and ``Little Oh" notations $\operatorname{O}$ and $o$ with their usual meanings.  
Let $F(x)$ and $G(x)$ be two real or complex-valued functions. We say that $F(x)=\BigO{G(x)}$ if we can find a positive real number $C$ not depending on $x$ such that $|F(x)|\le C |G(x)|$ for large values of $|x|$. We use ``$\sim$'' notation when two functions are asymptotically equivalent. As usual, we write $\mu(n)$, $\zeta(s)$, and $\zeta_K(s)$ for the M\"{o}bius function, the Riemann zeta function, and the Dedekind zeta function attached to a number field $K$, respectively. The symbol $\mathbb{F}_p$ denotes the field of residue classes modulo a prime number $p$.   

\section*{Preliminaries}
In this section, we collect some preliminary results needed in later arguments. 
We shall be using the following version of the Chinese Remainder Theorem for polynomials in $\mathbb{Z}[x]$.
\begin{proposition}\label{prop:CRT}
Let $P(x)$ be a polynomial in $\Z[x]$. For a positive integer $d$ with prime factorization
\[
d=p_1^{e_1}\cdots p_r^{e_r},
\]
the congruence $P(x)\equiv 0\Mod{d}$ is solvable if and only if the congruences $P(x)\equiv 0\Mod{p_i^{e_i}}$ are solvable for all $1\le i\le r$. Moreover, if $P(x)\equiv 0\Mod{p_i^{e_i}}$ has $N_i$ solutions, then the congruence $P(x)\equiv 0\Mod{d}$ has $N_1 N_2 \cdots N_r$ solutions. 
\end{proposition}
\begin{proof}
See, e.g., \cite[Theorem~4.13]{Ku1992}.
\end{proof}

In order to estimate the average of the function $f_{a_n,a_{n-1},\dots,a_1}(\ell)$, we will invoke a Tauberian Theorem given below.
\begin{proposition}\label{prop:Taub}
Let $\{h_n\}_{n\ge 1}$ be a strictly increasing and diverging real-valued sequence such that $h_1\ge 1$. Let  $\sum_{n=1}^\infty a_n h_n^{-s}$ be a Dirichlet series with non-negative coefficients and converges for $\Re(s)>1$ to a sum $f(s)$. Let $r\in \R\backslash\{-1,-2,-3,\dots\}$ such that 
\[
f(s)=(s-1)^{-r-1}g(s) + h(s),
\]
where $g$ and $h$ are holomorphic functions on some domain containing the closed half-plane $\Re(s)\ge 1$ and $g(1)\neq 0$. Then as $x\to \infty$, we have
\[
\sum_{h_n\le x} a_n \sim \frac{g(1)}{\Gamma(r+1)} x (\log x)^r.
\]
\end{proposition}
\begin{proof}
See, e.g., \cite[Lemma~4.1]{Bateman1981}.
\end{proof}

We also recall the following well known Abel's summation formula. See \cite[Theorem~4.2]{Apostol1976}.
\begin{proposition}\label{prop:Summbyparts}
For any arithmetical function $a(n)$ let
\[
A(x)=\sum_{n\le x}a(n),
\]
where $A(x)=0$ if $x<1$. Assume the function $b$ has a continuous derivative on the interval $[y,x]$, where $0<y<x$. Then we have
\[
\sum_{y<n\le x}a(n)b(n)=A(x)b(x)-A(y)b(y)-\int_{y}^x A(t)b'(t)~dt.
\]
\end{proposition}

\section{Proof of main results}\label{sec:lower}
The function which counts the number of solutions of a polynomial congruence equation has been studied before by Nagell \cite{TrygveNagel1921}, Sondor \cite{MR47679}, Hooley \cite{MR163874} \cite{MR637347}, and most recently by the authors in \cite{MR4088807} among others in different contexts, and mainly for polynomials which are irreducible or without multiple roots in $\mathbb{C}$. In the following lemma, we study this function for any polynomial $P(x)\in\mathbb{Z}[x]$. We will use this lemma in conjunction with Proposition \ref{prop:Taub} to prove our main result. 

\begin{lemma}\label{lem:pole} Let $P(x)\in \Z[x]$ and
\[
f_P(\ell)=|\{1\le a\le \ell | P(a)\equiv 0 \Mod{\ell}\}|.
\]
Then $f_P$ is multiplicative, and the Dirichlet series
\[
F(s)=\sum_{\ell=1}^\infty \frac{f_P(\ell)}{\ell^s}=\prod_{p \; \text{prime}} \left( 1+\sum_{k=1}^\infty \frac{f(p^k)}{p^{ks}}\right)
\]
is absolutely convergent in $\Re(s)>1$ and has a pole of order $J$ at $s=1$ where
\[
P(x)=\prod_{i=1}^J m_i(x)^{e_i} 
\]
is the factorization of $P(x)$ into irreducible polynomials $m_i(x)\in \Z[x]$.
\end{lemma}
\begin{proof}
The multiplicativity of $f_P$ follows from the Chinese Remainder Theorem for polynomial congruences given in Proposition~\ref{prop:CRT}. Let us assume that $P(x)$ is irreducible. We will show that $F(s)$ is absolutely convergent in $\Re(s)>1$ and has a simple pole at $s=1$. Let $D\neq 0$ be the discriminant of $P(x)$. We can break $F(s)$ as
\[
F(s)=\prod_{p\mid D}  \left( 1+\sum_{k=1}^\infty \frac{f_P(p^k)}{p^{ks}}\right)\prod_{p\nmid D}  \left( 1+\sum_{k=1}^\infty \frac{f_P(p^k)}{p^{ks}}\right).
\]
The first product above is a finite product since $D\neq 0$. By Lagrange's theorem \cite[Chapter 7]{MR2445243}, $P(x)$ has at most $\deg(P)$ roots modulo $p$ and so $f_P(p)\le \deg(P)$ for all primes $p$. Moreover, by Hensel's lifting lemma \cite[Chapter 7]{MR2445243}, for  primes $p$ not dividing $D$, a solution modulo $p$ has a unique lift modulo $p^k$ for every $k\ge 2$, hence $f_P(p^k)=f_P(p)\le \deg(P)$. For the primes $p$ dividing $D$, each solution modulo $p$ can have at most $p$ lifts modulo $p^2$ and so on, so that $f(p^k)\le \deg(P)p^{k-1}$ for the finitely many primes $p$ dividing $D$. Moreover in this case, the number of solutions $f(p^k)$ is bounded by $\deg(P)|D|^{1/2}$ (see, \cite{MR637347}). This ensures that the finite product is analytic and non-vanishing for $\Re(s)>0$. Combining these we have
\begin{align*}
F(s)=&\prod_{p\mid D}\left(1+\BigO{|p^{-s}|}\right)\prod_{p\nmid D} \left(1+f_P(p)\frac{p^{-s}}{1-p^{-s}}\right)\\
=&\prod_{p\mid D}\left(1+\BigO{|p^{-s}|}\right)\prod_{p\nmid D} \left(1+f_P(p)p^{-s}\right)\left(1+\BigO{|p^{-2s}|}\right).
\end{align*}
For primes $p$ not dividing the discriminant $D$, by an application of Dedekind's theorem (see, e.g., Theorem 5.5.1, \cite{MR2090972}), $f_P(p)$ is the number of prime ideals of ring of integers, $\mathcal{O}_K$, where $K=\Q[x]/(P(x))$, with residue field $\mathbb{F}_p$. Therefore, 
\begin{align*}
    \prod_{p\nmid D} \left(1+f_P(p)p^{-s}\right)&\left(1+\BigO{|p^{-2s}|}\right)\\ =&\prod_{p\nmid D}\left(\frac{1}{1-p^{-s}}\right)^{-f_P(p)}=\prod_{p\nmid D}\prod_{\substack{p\in \mathcal{P}\subset \mathcal{O}_K\\||\mathcal{P}||=p}}\frac{1}{1-||\mathcal{P}||^{-s}}\\
    =&\prod_{p\nmid D}\prod_{p\in \mathcal{P}\subset \mathcal{O}_K}\frac{1}{1-||\mathcal{P}||^{-s}}\prod_{\substack{p\in \mathcal{P}\subset \mathcal{O}_K\\||\mathcal{P}||\neq p}}(1-||\mathcal{P}||^{-s})\\
    =&\zeta_K(s)\prod_{p\mid D}\prod_{p\in \mathcal{P}\subset \mathcal{O}_K}(1-||\mathcal{P}||^{-s})\prod_p (1+\BigO{|p^{-2s}|})\\
    =&\zeta_K(s)\prod_{p\mid D}(1+\BigO{|p^{-s}|})\prod_p (1+\BigO{|p^{-2s}|}).
\end{align*}
where $||\mathcal{P}||$ denotes the norm of a prime ideal $\mathcal{P}$. It follows from above that
\[
F(s)=\zeta_K(s)\prod_{p\mid D}(1+\BigO{p^{-s}})\prod_p (1+\BigO{p^{-2s}}).
\]
The first product above is finite and non-vanishing everywhere in $\C$ while the second product is absolutely convergent and non-vanishing for $\Re(s)>1/2$. Since $\zeta_K$ has a simple pole at $s=1$, it follows that $F(s)$ is absolutely convergent in $\Re(s)>1$ with a simple pole at $s=1$.

Next, assume that $P(x)$ is reducible and has a factorization
\[
P(x)=\prod_{i=1}^J m_i(x)^{e_i} 
\]
in irreducible polynomials $m_i(x)\in \Z[x]$. We proceed along the lines of irreducible case with $D$ replaced by $\prod_{i=1}^J D_i$, where $D_i\neq 0$ is discriminant of $m_i(x)$. For primes $p$ not dividing $\prod_{i=1}^J D_i$, $m_i\Mod{p}$ are separable and coprime and hence, they do not have any root in common. Therefore, for primes $p$ not dividing $\prod_{i=1}^J D_i$
\[
f_P(p^\ell)=\sum_{i=1}^J f_{m_i}(p^\ell) \quad \text{for all } \ell\ge 1.
\]
Using the above and following the proof of the irreducible case, we arrive at
\[
F(s)=\prod_{i=1}^J \zeta_{K_i} (s)\prod_{p\mid \prod_{i=1}^J D_i}(1+\BigO{p^{-s}})\prod_p (1+\BigO{p^{-2s}})
\]
where $\zeta_{K_i}$ is the Dedekind zeta function of the number field $K_i=\Q[x]/(m_i(x))$. Since each of $\zeta_{K_i}$ has a simple pole at $s=1$, it follows that $F(s)$ is absolutely convergent in $\Re(s)>1$ with a pole at $s=1$ of order $J$.
\end{proof}

\begin{lemma} \label{meanf}
The function $f_{a_n,a_{n-1},\dots,a_1}$ in \eqref{fd} is multiplicative. As $x\rightarrow\infty$, we have
\[
\sum_{t\le x}f_{a_n,a_{n-1},\dots,a_1}(t)\sim Cx({\log x})^{J-1},
\]
where $J\ge 2$ is the number of distinct irreducible factors of the polynomial $a_nx^n+a_{n-1}x^{n-1}+\cdots+a_1x\in \Z[x]$.
\end{lemma}
\begin{proof}
Notice that $f_{a_n,a_{n-1},\dots,a_1}$ is same as $f_P$ in Lemma~\ref{lem:pole} for the polynomial $P(x)=a_nx^n+a_{n-1}x^{n-1}+\cdots+a_1x$. Using Lemma~\ref{lem:pole} for polynomial $a_nx^n+a_{n-1}x^{n-1}+\cdots+a_1x$, we find that the Dirichlet series for $f_{a_n,a_{n-1},\dots,a_1}(t)$ converges absolutely for $\Re(s)>1$ with a pole at $s=1$ of order $J$, where $J$ is the number of distinct irreducible factors of the polynomial $a_nx^n+a_{n-1}x^{n-1}+\cdots+a_1x$. By Tauberian Theorem in Proposition~\ref{prop:Taub}, the result follows. 
\end{proof}

\subsection{Proof of Theorem~\ref{thm:sub}}\label{sec:proof}
Let $(a,b)\in S(a_n,a_{n-1},\dots,a_1)$, then $\gcd(a(a_na^{n-1}+a_{n-1}a^{n-2}+\cdots+a_1),b)=1$ which implies $\gcd(a,b)=1$, and therefore, by definition, $(a,b)\in S(1)$. Moreover, if we choose $q_0=b/(a_na^{n}+a_{n-1}a^{n-1}+\cdots+a_1a)\in \Q^+$, then $(a,b)$ lies on the curve
\begin{equation}\label{eq:curve}
y=q_0(a_n x^n+a_{n-1} x^{n-1}+\cdots +a_1 x).
\end{equation}
Let $(a',b')\in \N^2$ be a point on the curve \eqref{eq:curve} between the origin and $(a,b)$, that is, $a'<a$. Then $a_na^{n}+a_{n-1}a^{n-1}+\cdots+a_1a$ must divide $b(a_n(a')^{n}+a_{n-1}(a')^{n-1}+\cdots+a_1a')$ but since
\[
\gcd(a_na^{n}+a_{n-1}a^{n-1}+\cdots+a_1a,b)=1
\]
it forces that $a_na^{n}+a_{n-1}a^{n-1}+\cdots+a_1a$ must divide $a_n(a')^{n}+a_{n-1}(a')^{n-1}+\cdots+a_1a'$ which is a contradiction since $a'<a$ implies
\[
a_n(a')^{n}+a_{n-1}(a')^{n-1}+\cdots+a_1a'<a_na^{n}+a_{n-1}a^{n-1}+\cdots+a_1a.
\]
Therefore, there is no point of $\N^2$ on the curve \eqref{eq:curve} between the origin and $(a,b)$. By Definition~\ref{def:vis},
$(a,b)\in V(a_n,a_{n-1},\dots,a_1)$. This proves the first part of Theorem~\ref{thm:sub}. For the second part,
consider 
\begin{align*}
|S(a_n,a_{n-1},\dots,a_1)\cap R_N|=\qquad \qquad \qquad \qquad \qquad \qquad \qquad \qquad \qquad \qquad &\\
|\{1\le a, b\le N \ | \  \gcd(a_na^n+a_{n-1}a^{n-1}+\cdots+a_1a,b)=1\}.
\end{align*}
To simplify the notation, let
\[
A_n:=a_na^n+a_{n-1}a^{n-1}+\cdots+a_1a.
\]
Using the M\"{o}bius identity,
\[
\sum_{d|n}\mu(n)=\begin{cases}
1 & n=1,\\
0 & \text{otherwise},
\end{cases}
\] we have
\begin{align}
    |S(a_n,a_{n-1},\dots,a_1)\cap R_N|&=\sum_{\substack{a\le N \\ b\le N}}\sum_{\substack{d|A_n\\d|b}}\mu(d)=\sum_{d\le N}\mu(d)\sum_{\substack{a\le N \\ d|A_n}}\sum_{b\le N/d}1 \nonumber\\
    &=\sum_{d\le N}\mu(d)\floor[\Big]{\frac{N}{d}}S_N(d)\nonumber\\
    &=\sum_{d\le N}\mu(d)\left(\frac{N}{d}S_N(d)+\BigO{S_N(d)}\right), \label{densS1}
\end{align}
where 
\begin{align}
   S_N(d)&=|\{1\le a\le N \ | \ A_n\equiv 0\Mod{d} \}|\nonumber\\&= 
   \frac{N}{d}f_{a_n,a_{n-1},\dots,a_1}(d)+\BigO{f_{a_n,a_{n-1},\dots,a_1}(d)}.
\end{align}
Here $f_{a_n,a_{n-1},\dots,a_1}(d)$ is as in \eqref{fd}. 
The above identity implies that in order to estimate \eqref{densS1}, we need to estimate the following two sums 
\[
S_1:=N^2\sum_{d\le N}\frac{\mu(d)f_{a_n,a_{n-1},\dots,a_1}(d)}{d^2},
\]
and
\[
S_2:=N\sum_{d\le N}\frac{\mu(d)f_{a_n,a_{n-1},\dots,a_1}(d)}{d}.
\]
The sum $S_2$ can be estimated using Abel's Summation Formula in Proposition~\ref{prop:Summbyparts}, where we take $a(n)=f_{a_n,a_{n-1},\dots,a_1}(n)$ and $b(n)=1/n.$ Then the partial sum $A(x)$ for $a(n)$ is estimated by Lemma~\ref{meanf}, and we obtain 
\[
S_2=\BigO{N\sum_{d\le N}\frac{f(d)}{d}}=\BigO{\frac{(\log N)^J}{N}}.
\]
By multiplicativity of $f_{a_n,a_{n-1},\dots,a_1}$ established in Lemma~\ref{meanf}, 
\[\lim_{N\rightarrow\infty}\frac{S_1}{N^2}=\prod_{p \text{ prime }}\left(1-\frac{f_{a_n,a_{n-1},\dots,a_1}(p)}{p^2}\right).\] Moreover, 
\[\lim_{N\rightarrow\infty}\frac{S_2}{N^2}=0.\]
Combining the above two estimates, we get the density result in Theorem~\ref{thm:sub}.

\subsection{Proof of Corollary~\ref{cor:densS1k}}\label{subsec:proof3}
Appealing to Theorem \ref{thm:sub}, the calculation for the density of $S(1,k)$ reduces to finding $f_{1,k}(p)$ for $k\ge 1$, where $p$ is prime.
For $k=1$, 
\[
f_{1,1}(p)=|\{1\le a\le p \ | \ p\mid a^2+a\}|=|\{p-1,p\}|=2
\]
for all primes $p$. Then the expression for $\operatorname{dens}(S(1,1))$ follows from Theorem \ref{thm:sub}.
For $k\geq 2$, we compute $f_{1,k}(p)$ when $p\mid k$ and $p\nmid k$, separately. When $p\nmid k$, then $k\equiv -r \Mod{p}$ for some $1\leq r\leq p-1$ which implies that $p\mid k+r$ and hence $f_{1,k}(p)=|\{p,r\}|=2$. When $p\mid k$, we have trivially, $f_{1,k}(p)=|\{p\}|=1$. Combining, we have
\begin{align*}
\operatorname{dens}(S(1,k))&=\prod_{\substack{p\mid k }}\Big(1-\frac{1}{p^2}\Big)\prod_{\substack{p\nmid k}}\Big(1-\frac{2}{p^2}\Big)\\
&=\prod_{\substack{p }}\Big(1-\frac{2}{p^2}\Big)\prod_{\substack{p\mid k}}\Big(1+\frac{1}{p^2-2}\Big).
\end{align*}
From Theorem~\ref{thm:sub}, $S(1,k)\subseteq V(1,k)$ and the result follows.
\section{Numerical results}\label{sec:numerical}
If a point $(a,b)$ is $\mathcal{F}(a_n, a_{n-1},\dots,a_1)$-visible, then there exists a $q\in \Q$ such that $b=q(a_n a^n + a_{n-1} a^{n-1}+\cdots+ a_1 a)$, that is, 
\begin{equation}\label{eq:q}
q = \frac{b}{a_n a^n + a_{n-1} a^{n-1}+\cdots+ a_1 a} .
\end{equation}
Therefore, every $\mathcal{F}(a_n, a_{n-1},\dots,a_1)$-visible point is associated to a rational of the form~\eqref{eq:q}. In fact, if
\begin{equation}\label{eq:qalphabeta}
    \Q_N(a_n, a_{n-1},\dots,a_1)=\Big\{\tfrac{b}{a_n a^n + a_{n-1} a^{n-1}+\cdots+ a_1 a} \ \Big | \ a,b \in \N, 1\le a,b\le N \Big\}
\end{equation}
denotes the set of rationals of the form~\eqref{eq:q}, one obtains the following result.
\begin{proposition}\label{lem:bij}
For a given vector $(a_n, a_{n-1},\dots,a_1)$ as in Definition \ref{def:vis}, 
\[
|V(a_n, a_{n-1},\dots,a_1)\cap R_N|=|\Q_N(a_n, a_{n-1},\dots,a_1)|,
\]
where $R_N$ is given in \eqref{eq:RN}.
\end{proposition}
\begin{proof}
A point $(a,b)\in V(a_n, a_{n-1},\dots,a_1)\cap R_N$ lies on a unique curve of the type 
\[
y = q (a_n x^{n}+a_{n-1}x^{n-1}+\cdots+a_1 x),
\]
where $q$ is as in \eqref{eq:q}. Hence, the result follows.
\end{proof}

Using Proposition~\ref{lem:bij} and Definition~\ref{def:density}, $\operatorname{dens}(V(a_n, a_{n-1},\dots,a_1))$ is given by
\begin{align}
    \operatorname{dens}(V(a_n, a_{n-1},\dots,a_1))=&\lim_{N\to \infty} \frac{|V(a_n, a_{n-1},\dots,a_1)\cap R_N|}{N^2} \nonumber\\
    =& \lim_{N\to \infty} \frac{|\Q_N(a_n, a_{n-1},\dots,a_1)|}{N^2}. \label{eq:densitynew}
\end{align}
To estimate $\operatorname{dens}(V(a_n, a_{n-1},\dots,a_1))$ numerically, for fixed $a_n, a_{n-1},\dots,a_1$ and a given $N$, we write a code to find $|\Q_N(a_n, a_{n-1},\dots,a_1)|$ and then approximate the density using \eqref{eq:densitynew}. Table~\ref{tab:density} provides the approximate densities of $V(a_n, a_{n-1},\dots,a_1)$ for different values of $a_n, a_{n-1},\dots,a_1$ calculated numerically. The visible lattice points along the polynomial families given in Table~\ref{tab:density} is shown in Figure~\ref{fig:vispoints}. One observes from Figure~\ref{fig:vispoints} that with increasing degree of the polynomial, non-visible points are rare. 

\begin{table}[ht!]
    \centering
    \begin{tabular}{|c|c|}
    \hline
        Polynomial & $|\Q_N(a_n, a_{n-1},\dots,a_1)|/N^2$  \\\hline
         $x^2 + x$ & $0.983$  \\ \hline 
         $2x^3 + 15x$ & $0.999$  \\ \hline 
          $6x^4 + 4x^2+13x$ &  $0.999$ \\ \hline 
           $7x^5 + 12x^4+4x^3+x^2+11x$ &  $0.999$ \\ \hline 
    \end{tabular}
    \caption{Approximate densities for some randomly chosen polynomial families. Here, $N=5000$.}
    \label{tab:density}
\end{table}

\begin{figure}[ht!]
  \begin{subfigure}{6cm}
    \centering\includegraphics[width=5cm]{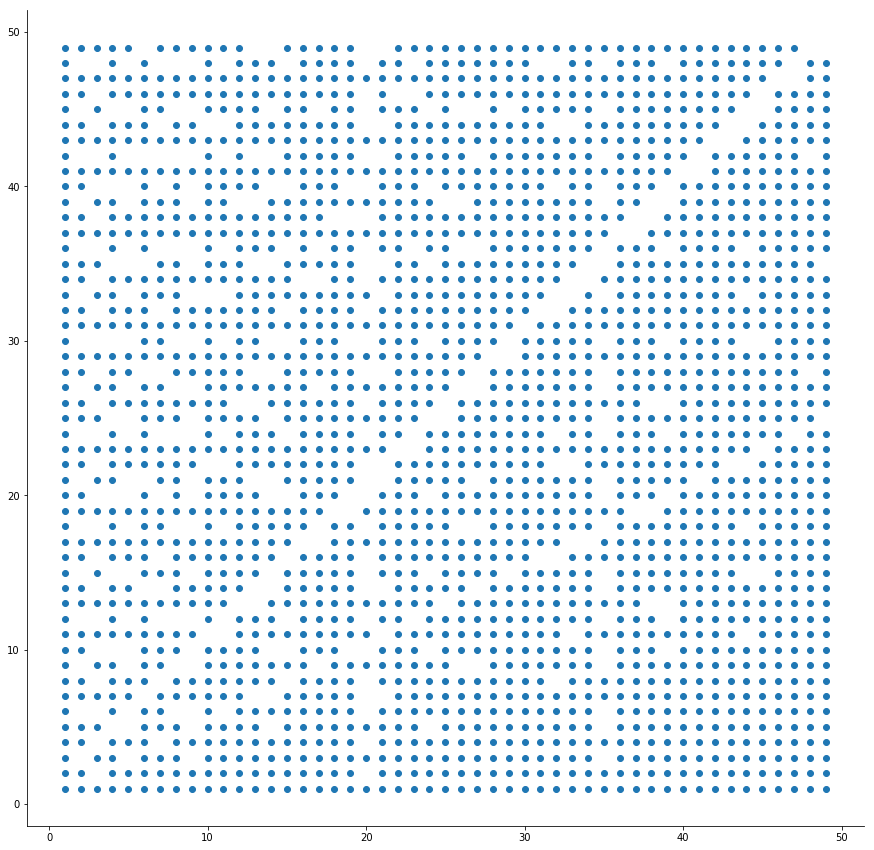}
    \caption{$V(1,1)\cap R_{50}$}
  \end{subfigure}
  \begin{subfigure}{6cm}
    \centering\includegraphics[width=5cm]{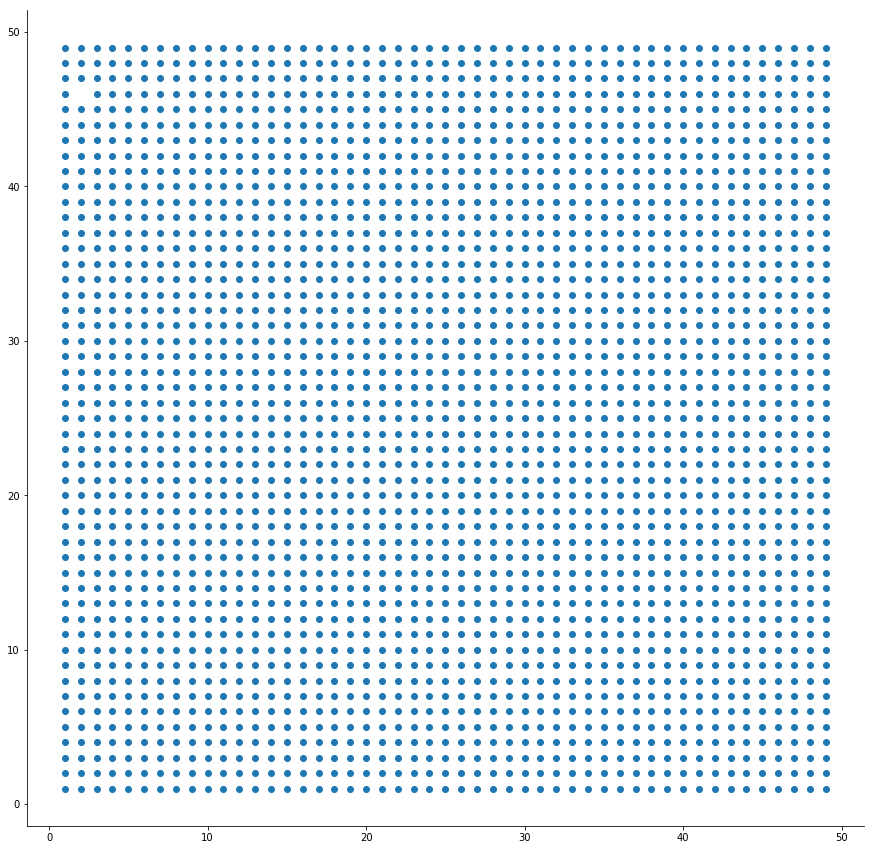}
    \caption{$V(2,0,15)\cap R_{50}$}
  \end{subfigure}
 
  \begin{subfigure}{6cm}
    \centering\includegraphics[width=5cm]{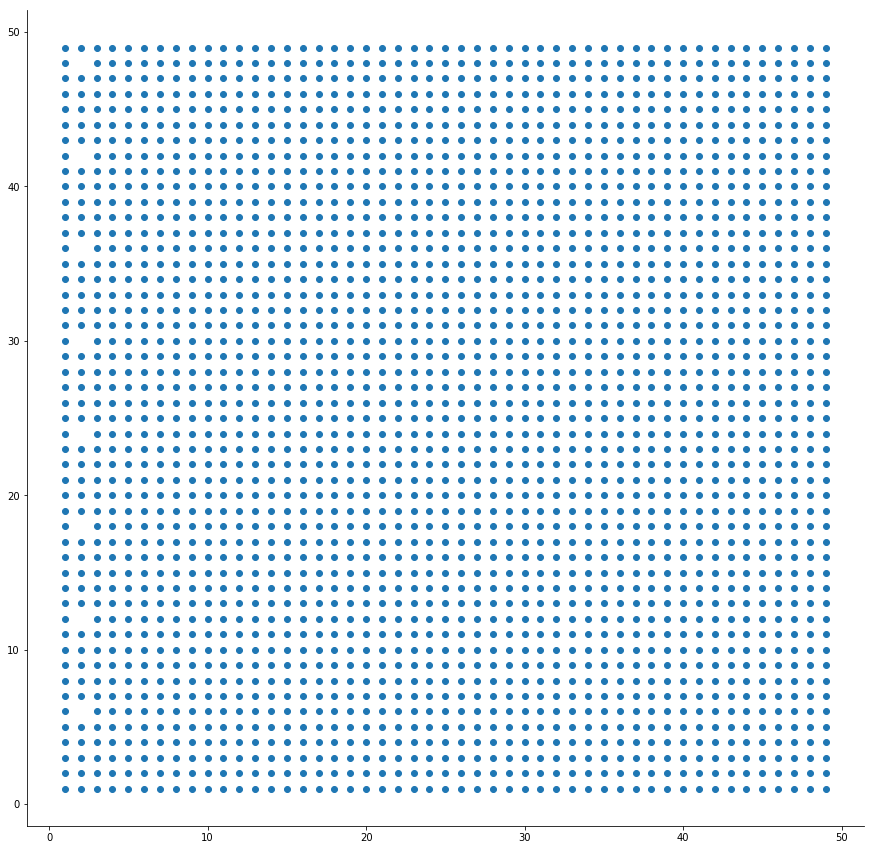}
    \caption{$V(6,0,4,13)\cap R_{50}$}
  \end{subfigure}
  \begin{subfigure}{6cm}
    \centering\includegraphics[width=5cm]{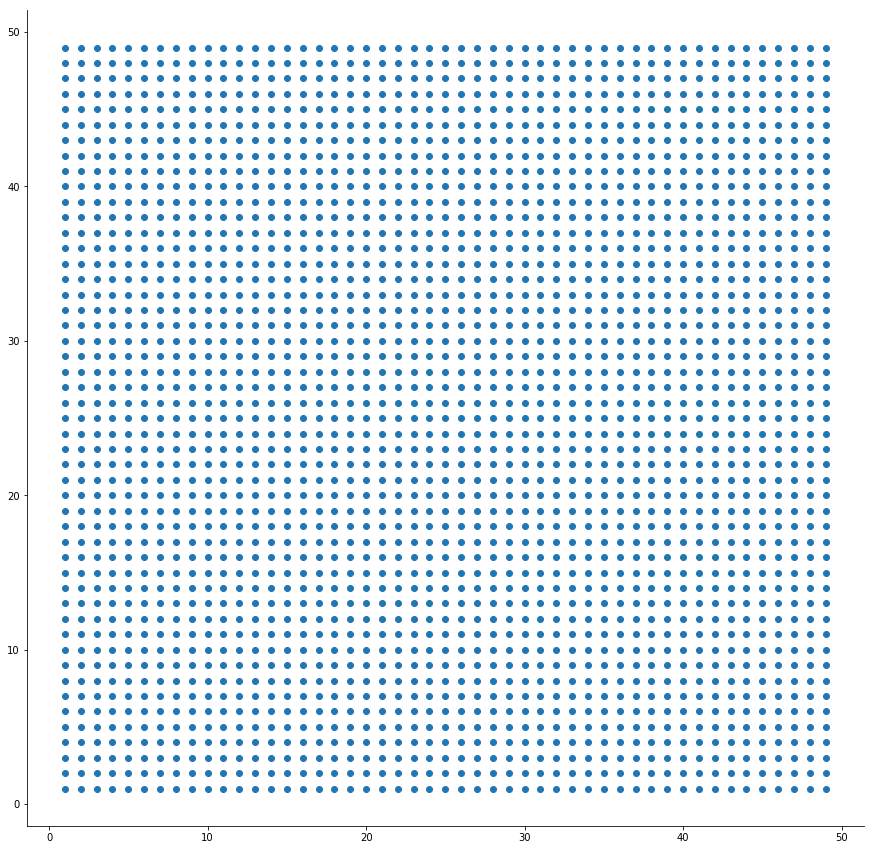}
    \caption{$V(7,12,4,1,11)\cap R_{50}$}
  \end{subfigure}
  \caption{Visible points along four different polynomial families in a $50\times 50$ grid.}
  \label{fig:vispoints}
\end{figure}


\section{Visibility along $\mathcal{F}(1,1)$}\label{sec:f11}
The family $\mathcal{F}(1,1)$ is particularly interesting as it is the first family (if one orders families by degree of polynomials and values of coefficients) with density conjectured to be one, see Conjecture~\ref{conj}. Using Proposition~\ref{lem:bij}, 
\[
\operatorname{dens}(V(1,1))=\lim_{N\to \infty}\frac{|\Q_N(1,1)|}{N^2}=\lim_{N\to \infty}\frac{\#\left\{\frac{b}{a(a+1)} \ \Big | \ 1\le a,b\le N\right\}}{N^2}.
\]
Therefore, it reduces to counting of distinct rationals of the form $b/a(a+1)$. A rational $b/a(a+1)$ will not be counted if
\[
\frac{b}{a(a+1)}=\frac{b'}{a'(a'+1)} \text{ for some } a'<a \text{ and } b'<b,
\]
For $a$ fixed, this implies that if
\[
b=b'\frac{a(a+1)}{a'(a'+1)} \text{ for some } a'<a \text{ and } b'<b,
\]
then $b/a(a+1)$ will be excluded.
This further implies that $b/a(a+1)$ will not be counted if $b$ is a multiple of 
\[
\frac{a(a+1)}{\gcd(a(a+1),a'(a'+1))} \text{ for some } a'<a.
\]
Table~\ref{tab:missingbs} provides values of $b$ whose multiples are missing for a given $a$. 

\begin{table}[]
    \centering
    \begin{tabular}{|c|c||c|c||c|c||c|c|}\hline
        $a$ & $b$ & $a$ & $b$ & $a$ & $b$ & $a$ & $b$\\ \hline
         1 & \text{none} & 6 & 7 & 11 & 6, 11 & 16 & 17\\ \hline
          2 & 3 & 7 & 4 & 12 & 13 & 17 & 9, 17\\ \hline
           3 & 2 & 8 & 6, 9 & 13 & 7, 13 & 18 & 19\\ \hline
            4 & 5 & 9 & 3, 5 & 14 & 5, 7 & 19 & 10, 19\\ \hline
             5 & 3, 5 & 10 & 11 & 15 & 8, 10, 12 & 20 & 2, 7, 15\\ \hline
    \end{tabular}
    \caption{Multiples of $b$ which will be missing for the given value of $1\le a\le 20$.}
    \label{tab:missingbs}
\end{table}

Let $1\le a,b \le N$ for some $N\in \N$. For a fixed $a$, let $b_1<b_2<\dots<b_{k_a}$ be the values of $b$ such that $(a, l b_i)\notin V(1,1)$ for $l\in \N$, $1\le l \le N/b_i$ and $i=1,2,\dots,k_a$. Then, using inclusion-exclusion principle, for a fixed $a$, the number of $b$'s such that $(a,b)\notin V(1,1)$ is equal to
\begin{equation}\label{eq:nan}
n(a, N)=\sum_{i=1}^{k_a} \floor*{\frac{N}{b_i}}-\sum_{j=2}^{k_a}\sum_{i=1}^{j-1} \floor*{\frac{N}{\lcm(b_i, b_j)}}+\cdots +(-1)^{k_a-1} \floor*{\frac{N}{\lcm(b_1, b_2,\dots,b_{k_a})}}.
\end{equation}
The Conjecture~\ref{conj} for $V(1,1)$ is equivalent to
\begin{equation}\label{eq:conjeq}
\sum_{a=1}^N n(a,N) = o(N^2).
\end{equation}
Let $c(a):=n(a,N)/N$ for all $a$ and $N$. Then, \eqref{eq:conjeq} is equivalent to
\begin{equation}\label{eq:conjeq1}
    \sum_{a=1}^N c(a) = o(N).
\end{equation}
Figure~\ref{fig:cvsa} shows plots of $c(a)$ for different ranges of $a$. 
\begin{figure}[ht!]
  \begin{subfigure}{6cm}
    \centering\includegraphics[width=5cm]{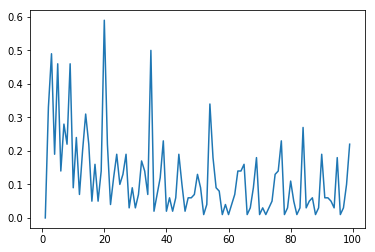}
    \caption{$1\le a \le 100$}
  \end{subfigure}
  \begin{subfigure}{6cm}
    \centering\includegraphics[width=5cm]{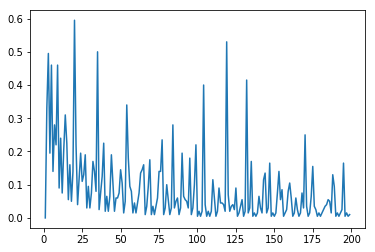}
    \caption{$1\le a \le 200$}
  \end{subfigure}
 
  \begin{subfigure}{6cm}
    \centering\includegraphics[width=5cm]{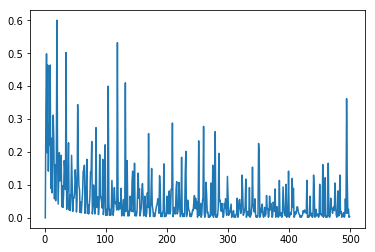}
    \caption{$1\le a \le 500$}
  \end{subfigure}
  \begin{subfigure}{6cm}
    \centering\includegraphics[width=5cm]{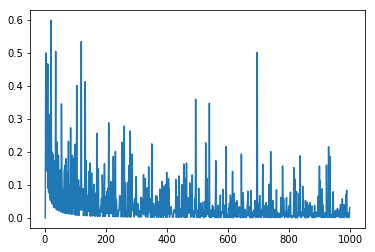}
    \caption{$1\le a \le 1000$}
  \end{subfigure}
  \caption{$c(a)$ vs. $a$}
  \label{fig:cvsa}
\end{figure}
Since the first term in $n(a,N)$ in \eqref{eq:nan} dominates, we have
\begin{align*}
    c(a)\le &\frac{1}{N}\sum_{i=1}^{k_a}\floor*{\frac{N}{b_i}}\le \sum_{i=1}^{k_a}\frac{1}{b_i}=:c_1(a).
\end{align*}
Figure~\ref{fig:sumc1} depicts the  average sum of $c_1(a)$ for initial $500$ values. The figure along with \eqref{eq:conjeq1} suggests the following conjecture. 

\begin{conjecture}\label{conj2}
The $\operatorname{dens(V(1,1))}$ is equal to $1$ if
\begin{equation*}
    \sum_{a=1}^N c_1(a) = o(N).
\end{equation*}
\end{conjecture}

\begin{figure}[ht!]
    \centering
    \includegraphics[scale=0.6]{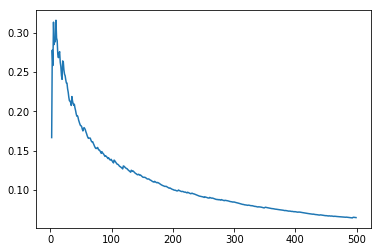}
    \caption{$\frac{1}{N}\sum_{a=1}^N c_1(a)$ vs. $N$}
    \label{fig:sumc1}
\end{figure}

\section*{Acknowledgements}
The authors are indebted to M. Ram Murty for his valuable inputs and guidance during the preparation of this article. SC and AKP are both supported by the Science and Engineering Research Board, Department of Science and Technology, Government of India under grants SB/S2/RJN-053/2018, and SRG/2019/000741, respectively.

\section*{\refname}

\bibliography{references1.bib}

\begin{thebibliography}{10}
\expandafter\ifx\csname url\endcsname\relax
  \def\url#1{\texttt{#1}}\fi
\expandafter\ifx\csname urlprefix\endcsname\relax\def\urlprefix{URL }\fi
\expandafter\ifx\csname href\endcsname\relax
  \def\href#1#2{#2} \def\path#1{#1}\fi

\bibitem{Apostol1976}
T.~M. Apostol, {Introduction to Analytic Number Theory}, Springer New York,
  1976.

\bibitem{Baker2010PrimitiveDomains}
R.~C. Baker, {Primitive lattice points in planar domains}, Acta Arithmetica
  142~(3) (2010) 267--302.
\newblock \href {http://dx.doi.org/10.4064/aa142-3-4}
  {\path{doi:10.4064/aa142-3-4}}.

\bibitem{Boca2000DistributionOrigin}
F.~P. Boca, C.~Cobeli, A.~Zaharescu, {Distribution of lattice points visible
  from the origin}, Communications in Mathematical Physics 213~(2) (2000)
  433--470.
\newblock \href {http://dx.doi.org/10.1007/s002200000250}
  {\path{doi:10.1007/s002200000250}}.

\bibitem{Huxley1996PrimitiveBy}
M.~N. Huxley, W.~G. Nowak, {Primitive lattice points in convex planar domains
  by}, Acta Arithmetica 76~(3) (1996) 271--283.
\newblock \href {http://dx.doi.org/10.4064/aa-76-3-271-283}
  {\path{doi:10.4064/aa-76-3-271-283}}.

\bibitem{Baake2015}
M.~Baake, C.~Huck, {Ergodic properties of visible lattice points}, Proceedings
  of the Steklov Institute of Mathematics 288~(1) (2015) 165--188.
\newblock \href {http://dx.doi.org/10.1134/S0081543815010137}
  {\path{doi:10.1134/S0081543815010137}}.

\bibitem{Baake2000}
M.~Baake, R.~V. Moody, P.~A. Pleasants, {Diffraction from visible lattice
  points and kth power free integers}, Discrete Mathematics 221~(1-3) (2000)
  3--42.
\newblock \href {http://dx.doi.org/10.1016/S0012-365X(99)00384-2}
  {\path{doi:10.1016/S0012-365X(99)00384-2}}.

\bibitem{Baake2017OnDensity}
M.~Baake, C.~Huck, N.~Strungaru, {On weak model sets of extremal density},
  Indagationes Mathematicae 28~(1) (2017) 3--31.
\newblock \href {http://dx.doi.org/10.1016/j.indag.2016.11.002}
  {\path{doi:10.1016/j.indag.2016.11.002}}.

\bibitem{Cilleruelo2019}
J.~Cilleruelo, J.~L. Fern{\'{a}}ndez, P.~Fern{\'{a}}ndez,
  \href{https://doi.org/10.1016/j.ejc.2018.08.004}{{Visible lattice points in
  random walks}}, European Journal of Combinatorics 75 (2019) 92--112.
\newblock \href {http://dx.doi.org/10.1016/j.ejc.2018.08.004}
  {\path{doi:10.1016/j.ejc.2018.08.004}}.
\newline\urlprefix\url{https://doi.org/10.1016/j.ejc.2018.08.004}

\bibitem{Goins2018}
E.~H. Goins, P.~E. Harris, B.~Kubik, A.~Mbirika, {Lattice Point Visibility on
  Generalized Lines of Sight}, American Mathematical Monthly 125~(7) (2018)
  593--601.
\newblock \href {http://dx.doi.org/10.1080/00029890.2018.1465760}
  {\path{doi:10.1080/00029890.2018.1465760}}.

\bibitem{Ku1992}
K.~Conrad,
  \href{https://kconrad.math.uconn.edu/blurbs/ugradnumthy/crt.pdf}{{THE CHINESE
  REMAINDER THEOREM}}, Tech. rep., University of Connecticut (1992).
\newline\urlprefix\url{https://kconrad.math.uconn.edu/blurbs/ugradnumthy/crt.pdf}

\bibitem{Bateman1981}
P.~T. Bateman, P.~Erd{\"{o}}s, C.~Pomerance, E.~G. Straus, {The arithmetic mean
  of the divisors of an integer}, in: Analytic Number Theory, Springer, Berlin,
  Heidelberg, 1981, pp. 197--220.
\newblock \href {http://dx.doi.org/10.1007/bfb0096462}
  {\path{doi:10.1007/bfb0096462}}.

\bibitem{TrygveNagel1921}
T.~Nagel, \href{http://eudml.org/doc/234957}{{G{\'{e}}n{\'{e}}ralisation d'un
  th{\'{e}}or{\`{e}}me de Tchebycheff}}, Journal de Math{\'{e}}matiques Pures
  et Appliqu{\'{e}}es 4 (1921) 343--356.
\newline\urlprefix\url{http://eudml.org/doc/234957}

\bibitem{MR47679}
G.~S{\'{a}}ndor, \href{https://doi.org/10.1007/BF02392280}{{Uber die
  {\{}A{\}}nzahl der {\{}L{\}}{\"{o}}sungen einer {\{}K{\}}ongruenz}}, Acta
  Math. 87 (1952) 13--16.
\newblock \href {http://dx.doi.org/10.1007/BF02392280}
  {\path{doi:10.1007/BF02392280}}.
\newline\urlprefix\url{https://doi.org/10.1007/BF02392280}

\bibitem{MR163874}
C.~Hooley, \href{https://doi.org/10.1112/S0025579300003466}{{On the
  distribution of the roots of polynomial congruences}}, Mathematika 11 (1964)
  39--49.
\newblock \href {http://dx.doi.org/10.1112/S0025579300003466}
  {\path{doi:10.1112/S0025579300003466}}.
\newline\urlprefix\url{https://doi.org/10.1112/S0025579300003466}

\bibitem{MR637347}
M.~N. Huxley, {A note on polynomial congruences}, in: Recent progress in
  analytic number theory, {\{}V{\}}ol. 1 ({\{}D{\}}urham, 1979), Academic
  Press, London-New York, 1981, pp. 193--196.

\bibitem{MR4088807}
V.~Cri{\textbackslash}csan, P.~Pollack, {The smallest root of a polynomial
  congruence}, Math. Res. Lett. 27~(1) (2020) 43--66.

\bibitem{MR2445243}
G.~H. Hardy, E.~M. Wright, {An introduction to the theory of numbers}, sixth
  Edition, Oxford University Press, Oxford, 2008.

\bibitem{MR2090972}
M.~R. Murty, J.~Esmonde, {Problems in algebraic number theory}, 2nd Edition,
  Vol. 190 of Graduate Texts in Mathematics, Springer-Verlag, New York, 2005.

\end{thebibliography}

\end{document}